\documentclass[twocolumn]{IEEEtran} 

\title{Fast Digital Convolutions using Bit-Shifts}
\author{Shekhar S. Chandra\thanks{Shekhar S. Chandra is with the School of Physics, Monash University, Australia. Email:~\href{mailto:Shekhar.Chandra@monash.edu}{Shekhar.Chandra@monash.edu}}}
\date{\today}

\usepackage[breaklinks,plainpages=false,pdfpagelabels]{hyperref}
\hypersetup{colorlinks,citecolor=black,filecolor=black,linkcolor=black,urlcolor=black}

\usepackage[australian]{babel}
\usepackage[T1]{fontenc} 
\usepackage{amsmath} 
\usepackage{amssymb}
\usepackage{amsthm}
\usepackage{mathrsfs} 
\usepackage{textcomp} 
\usepackage[numbers,square,compress]{natbib}
\usepackage{graphicx}
\usepackage[caption=false]{subfig}
\usepackage{microtype} 
\usepackage{mdwmath} 
\usepackage{mdwtab}
\usepackage{fixltx2e} 
\usepackage{array}
\usepackage{nicefrac}
\usepackage{fixmath}
\usepackage{acronym} 

\makeatletter
\def\imod#1{\allowbreak\mkern10mu({\operator@font mod}\,#1)}
\makeatother

\newtheorem{theorem}{Proposition}

\newcommand{\primeSymbol}{p}
\newcommand{\modulusSymbol}{m}
\newcommand{\FermatSymbol}{F}
\newcommand{\firstIndex}{j}
\newcommand{\secondIndex}{k}
\newcommand{\eqnTag}{Eq.}
\newcommand{\eqnsTag}{Eqs}
\newcommand{\figTag}{Fig.}
\newcommand{\figsTag}{Figs}
\newcommand{\tabTag}{Table}
\newcommand{\secTag}{Sec.}

\newcommand{\propTag}{Prop.}

\begin{document}

\maketitle
\acresetall

\begin{abstract}
An exact, one-to-one transform is presented that not only allows digital circular convolutions, but is free from multiplications and quantisation errors for transform lengths of arbitrary powers of two. The transform is analogous to the \acl{DFT}, with the canonical harmonics replaced by a set of cyclic integers computed using only bit-shifts and additions modulo a prime number. The prime number may be selected to occupy contemporary word sizes or to be very large for cryptographic or data hiding applications. The transform is an extension of the \aclp{RT} via Carmichael's Theorem. These properties allow for exact convolutions that are impervious to numerical overflow and to utilise \acl{FFT} algorithms.

\end{abstract}

\begin{IEEEkeywords}
DSP-FAST; Number Theoretic Transform; Discrete Fourier Transform; Fast Fourier Transform; Fermat Number Transform.
\end{IEEEkeywords}

\acresetall 
\section{Introduction}\label{sec::Intro}
The \ac{DFT} is commonly used to compute the circular convolution $h$ of two finite (or periodic) sequences $f$ and $g$ of length $N$ as
\begin{equation}\label{eqn::CCP}
 h(\firstIndex) = f(\firstIndex) * g(\firstIndex) = \sum^{N-1}_{\secondIndex=0} f(\secondIndex) \cdot g(\firstIndex-\secondIndex),
\end{equation}
by using the Convolution Theorem, where \eqnTag~\eqref{eqn::CCP} can be computed simply as a product of both sequences in Discrete Fourier space. This theorem provides a computational advantage because the Cooley-Tukey algorithm~\citep{Cooley1965} for computing the \ac{DFT} has a computational complexity of $O(N\log_2 N)$, as opposed to $O(N^2)$ for direct methods, when $N$ is a power of two.

A major result of this letter regarding convolutions can be summarised as follows. Let $\langle a \rangle_\modulusSymbol$ denote computing the remainder with respect to $\modulusSymbol$ (see Appendix~\ref{sec::Sino} for details), where $a \in \mathbb{Z}$, i.e. $a$ is an integer, and $\modulusSymbol$ is a prime number (or prime) as given in, but not restricted to, \tabTag~\ref{tab::Primes}. To compute the digital circular convolution of two finite integer sequences, one transforms both sequences as
\begin{equation}\label{eqn::NTT}
 X(u) = \sum_{t=0}^{N-1} \left\langle x(t)\cdot 2^{ut} \right\rangle_\modulusSymbol,
\end{equation}
which only involves bit-shifting, modulo and addition operations. The coefficients of these two sequences are multiplied and the result is inverted as
\begin{equation}\label{eqn::iNTT}
 x(t) = \frac{1}{N} \sum_{u=0}^{N-1} \left\langle X(u)\cdot 2^{-ut} \right\rangle_\modulusSymbol,
\end{equation}
where $2^{-ut} = 2^{\langle -ut \rangle_N} = 2^{N-ut}$. Note that the convolution is free from round-off errors as no floating-point numbers are required. Exact digital filtering involving division operations can be performed via multiplicative inverses, i.e. an integer $N^{-1}$ so that $\langle N \cdot N^{-1} \rangle_\modulusSymbol = 1$. The Cooley-Tukey algorithm~\citep{Cooley1965} is easily applied by replacing $e^{2\pi i\alpha/N}$ with the powers of two $2^\alpha$, where $\alpha\in \mathbb{Z}$ and $i^2=-1$. In other words, the $N^{\text{th}}$ root of unity $e^{2\pi i N} = 1$ is replaced with the integer-only equivalent $\langle 2^N \rangle_\modulusSymbol = 1$. The transform lengths permitted are $N=2^n$ when $n\in \mathbb{Z}$ that divide $N_{\text{Max}}$ in \tabTag~\ref{tab::Primes}. For example, the prime 13631489 allows for all $N$ up to and including $2^{19}$.
\begin{figure}[htbp]
 \centering
 \subfloat[]{
 \includegraphics[width=0.15\textwidth]{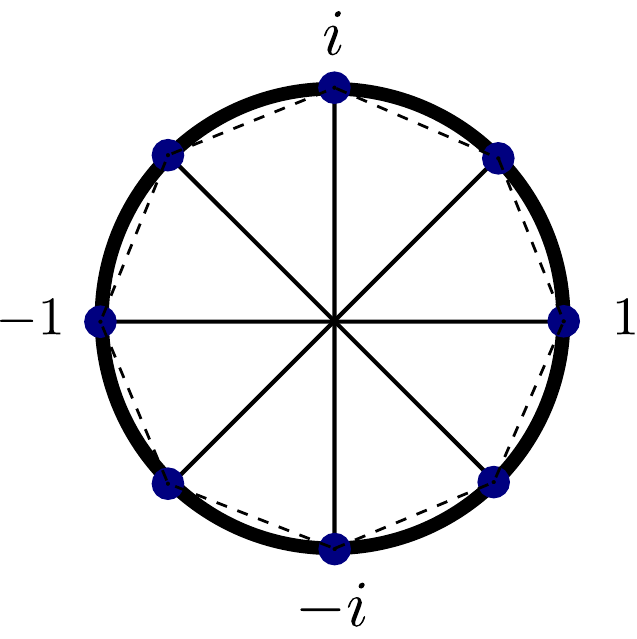}} 
 \hspace{0.75cm}
 \subfloat[]{
 \includegraphics[width=0.15\textwidth]{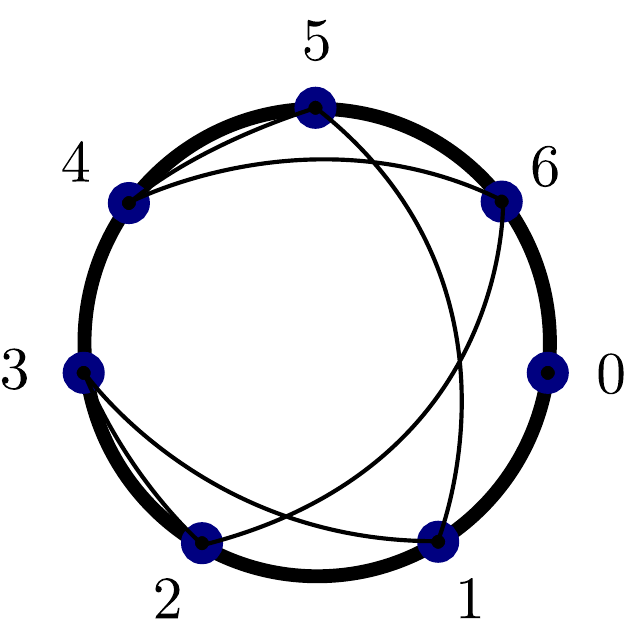}}
 
 \subfloat[]{
 \includegraphics[width=0.15\textwidth]{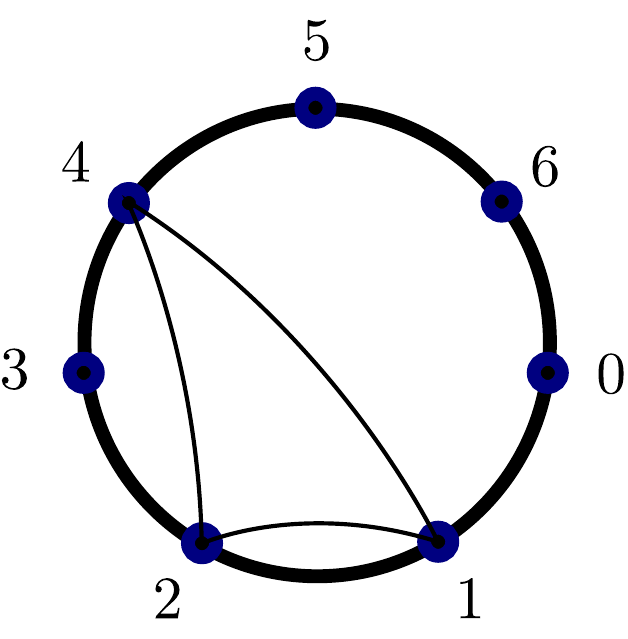}}
 \hspace{0.75cm}
 \subfloat[]{
 \includegraphics[width=0.15\textwidth]{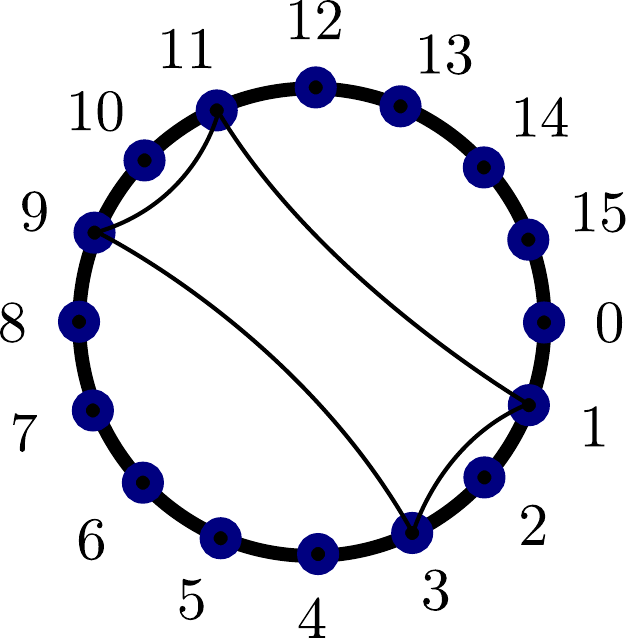}}
 \caption{The unit circle (a) and the digital circles for computing exact convolutions (b)-(d). The paths in (b)-(d) show the integers produced from successive powers of 3 modulo 7 (resulting in $\{1,3,2,6,4,5,1,\ldots\}$) in (b), successive powers of 2 modulo 7 (resulting in $\{1,2,4,1,2,4,1,\ldots\}$) in (c) and successive powers of 3 modulo 16 (resulting in $\{1,3,9,11,1,\ldots\}$) in (d). The transform lengths permissible for each are $N_{\text{Max}}=6$ (and its divisors) for (b), $N_{\text{Max}}=3$ for (c) and $N_{\text{Max}}=4$ for (d).}
 \label{fig::Circles}
\end{figure}
\begin{table}[htbp]
	\centering 
	\begin{tabular}{c c c c} 
	\hline\hline 
	\\[-1.5ex]
	$m$ & Max. Transform & Corresponding & Word \\[-1ex]
	\raisebox{1.75ex}{Prime} & Length $N_{\text{Max}}$ & Fermat Number & Size  \\[0.5ex]
	\hline 
	\\[-1.5ex]
	641 & 64 & $\FermatSymbol_5$ & 16-bit\\[0.5ex]
	2424833 & 1024 & $\FermatSymbol_9$ & 32-bit\\[0.5ex]
	319489 & 4096 & $\FermatSymbol_{11}$ & 32-bit\\[0.5ex]
	13631489 & 524288 & $\FermatSymbol_{18}$ & 32-bit\\[0.5ex]
	\hline 
	\hline 
	\end{tabular}
	\caption{Prime values of the modulus $\modulusSymbol$ for different transform lengths $N$. Lengths permissible are $2^n \leqslant N_{\text{Max}}$. Note that the size of the prime is not necessarily proportional to $N_{\text{Max}}$.\label{tab::Primes}}
\end{table}

Equations~\eqref{eqn::NTT} and \eqref{eqn::iNTT} are an extension of the \acp{RT}~\citep{Agarwal1974}, which, until now, were only practical for small ($N \leqslant 128$) transform lengths, severely limiting its applications~\citep{Agarwal1974a}. The theory developed in \secTag~\ref{sec::Carmichael} and~\ref{sec::Transforms} of this letter applies Carmichael's Theorem to remove these limitations completely by generalising the concept of an integer based $N^{\text{th}}$ root of unity. \tabTag~\ref{tab::Primes} shows that the moduli $\modulusSymbol$ chosen for very large transform lengths $N$ can fit into a 32-bit word size because the primes are shown to be any (including the smallest) one of the factors of large Fermat numbers, which are numbers of the form
\begin{equation}\label{eqn::FermatNumbers}
 F_n = 2^{2^n} + 1.
\end{equation}
\secTag~\ref{sec::Transforms} also presents a modulus free transform similar to \eqnsTag~\eqref{eqn::NTT} and \eqref{eqn::iNTT}, i.e. an integer-only transform without the need for modulo operations, using the same theory.

The preservation of the \ac{CCP}, which allows one to use the Convolution Theorem for finite sequences, is made possible because the unit circle of the \ac{DFT} (see \figTag~\ref{fig::Circles}(a)) is replaced with the digital ``circle''
\begin{eqnarray}\label{eqn::ChineseHypothesis}
2^{N} &\equiv& 1 \imod \modulusSymbol,
\end{eqnarray}
so that $2^{N} - 1$ is a multiple of $\modulusSymbol$ (see \figsTag~\ref{fig::Circles}(c) and~\ref{fig::Congruence}). The successive powers $\{1,\ldots,N\}$ of two generates a unique set of integers in some order modulo $\modulusSymbol$. The result is a circle consisting of a set of cyclic integers, i.e. a set of unique integers with a period. These integers define the ``harmonics'' of the transform in \eqnsTag~\eqref{eqn::NTT} and \eqref{eqn::iNTT}. Integer coefficients allow computations to be done without round-off error or numerical overflow, since the results are congruent modulo $\modulusSymbol$~\citep{Nussbaumer1978}.

The new transform is an extension of the \acp{RT}~\citep{Agarwal1974}, which consist of the \ac{FNT}~\citep{Schonhage1971} and the \ac{MNT}~\citep{Rader1972}. When using bit-shifts, the \ac{FNT} and the \ac{MNT} only utilise moduli 64-bits or less in size for small ($N \leqslant 128$) transform lengths. The \ac{FNT} only allows larger transform lengths when not in bit-shift form, i.e. the sequence is multiplied by powers of integers other than two (such as three) in \eqnsTag~\eqref{eqn::NTT} and~\eqref{eqn::iNTT}. A solution is to use multi-dimensional techniques which provide a limited extension of the transform lengths~\citep{Agarwal1974a}. Pollard~\citep{Pollard1971} showed that these transforms may have an alternate form (that also does not utilise bit-shifting) via Euler's Theorem
\begin{equation}\label{eqn::Euler}
a^{N} \equiv a^{\phi(\modulusSymbol)} \equiv 1 \imod \modulusSymbol,
\end{equation}
where $a, \modulusSymbol \in \mathbb{Z}$, $a$ and $\modulusSymbol$ are coprime and $\phi(\modulusSymbol)$ is provided in Appendix~\ref{sec::Totient}. When $\modulusSymbol$ is prime, \eqnTag~\eqref{eqn::Euler} becomes $a^{\modulusSymbol-1} \equiv 1 \imod \modulusSymbol$, which is known as Fermat's (little) Theorem. These alternate transforms also preserve the \ac{CCP} while allowing arbitrary transform lengths and are referred to as \acp{NTT}. The modulus $\modulusSymbol$ for these transforms are primes of the form $\primeSymbol = k\cdot N + 1$ so that $\modulusSymbol-1 = k\cdot N$, allowing transform lengths of $N = 2^n$ and divisors of $N$~\citep{Bhattacharya2000}. On the rare occasions when computations explicitly require large exponents in \eqnsTag~\eqref{eqn::ChineseHypothesis} and~\eqref{eqn::Euler}, they can still be computed in logarithmic complexity using modular exponentiation methods.

The \acp{NTT} are ideal for real data but can also be complex-valued~\citep{Nussbaumer1976}. \acp{NTT} have been applied to fast multiplication of very large integers~\citep{Schonhage1971}, fast digital convolutions and filtering~\citep{Agarwal1974,Agarwal1974a}, encryption~\citep{Mohan1985} and discrete Radon transforms~\citep{Chandra2010c}. Agarwal and Burrus~\citep{Agarwal1974} showed the \ac{NTT} to be faster than the \ac{FFT} in their implementation. Chandra~\citep{Chandra2010c} (via the open-source library~\citep{Chandra2009b}) showed that a modern implementation of the \ac{NTT} outperforms the popular FFTW library.

\section{Carmichael's Theorem}\label{sec::Carmichael}
This section presents a new and more general theory of \acfp{NTT} utilising the concept of primitive roots from Carmichael's Theorem~\citep{Carmichael1910} (see Appendix~\ref{sec::Totient}), a generalisation of Euler's Theorem given in \eqnTag~\eqref{eqn::Euler}. The primary result of this new theory are \eqnsTag~\eqref{eqn::NTT} and~\eqref{eqn::iNTT} when using \tabTag~\ref{tab::Primes}.

To construct an \ac{NTT}, one needs a set of unique cyclic integers sufficient to represent all the coefficients of a given transform length $N$. In Euler's Theorem, the integer $a$ is a special integer called a primitive root (or $\phi$-root~\citep{Carmichael1910}) related to the modulus, where successive powers $\{1,\ldots,\modulusSymbol-1\}$ of $a$ generates all the integers $\{1,\ldots,\modulusSymbol-1\}$ in some unique order modulo $\modulusSymbol$ (see \figTag~\ref{fig::Circles}(b)). This condition works well, but is very restrictive as not all integers are $\phi$-roots of a given modulus and not all moduli have $\phi$-roots. For example, the integer 2 is only a $\phi$-root for primes of the form $4\cdot q + 1$ when $q$ itself is prime~\citep[pg. 102]{Beiler1966}. Thus, the integer 2 is only suitable for prime length \acp{NTT} and not a $\phi$-root of primes of the form $k\cdot 2^n + 1$ required for power of two transform lengths in this theory.

Carmichael~\citep{Carmichael1910} developed the concept of the primitive $\lambda$-root, where the successive powers $\{1,\ldots,\nu\}$ of this root generates a fixed subset of the integers $\{1,\ldots,\modulusSymbol-1\}$ in some unique order modulo $\modulusSymbol$ (see \figTag~\ref{fig::Circles}(c)). The number of integers in this subset is $\nu$, where $\nu$ is the smallest integer for which
\begin{equation}\label{eqn::LambdaRoot}
 a^{\nu} \equiv 1 \imod \modulusSymbol,
\end{equation}
is true. Thus, one gets a set of unique cyclic integers of order $\nu$ capable of representing $\nu$ distinct coefficients. Carmichael~\citep{Carmichael1910} points out the smallest composite (non-prime) modulus $\modulusSymbol$ for when this and \eqnTag~\eqref{eqn::ChineseHypothesis} is true is $\modulusSymbol=341$. Since $341= 11\cdot 31$ and $2^{10}-1 = 3 \cdot 11 \cdot 31$, then $\nu = 10$ as $ 2^{10} \equiv 2^{340} \equiv 1 \imod {341}$. Such composite moduli are now known as Poulet numbers. To construct a unique and sufficient set of coefficients for power of two transform lengths, one needs to show that $\nu = 2^n$ and find a modulus $\modulusSymbol$ so that 2 is a $\lambda$-root of $\modulusSymbol$.

\section{New Transforms}\label{sec::Transforms}
This section presents the proof of the transform stated in \eqnsTag~\eqref{eqn::NTT} and~\eqref{eqn::iNTT}. It will be shown that the primes in \tabTag~\ref{tab::Primes} for this transform can be chosen to be any (including the smallest) prime factor of the Fermat numbers~\eqref{eqn::FermatNumbers}. The section will conclude with a discussion of another useful result.

\subsection{Multiplication Free}
In order to satisfy \eqnTag~\eqref{eqn::ChineseHypothesis}, a prime modulus must be selected so that
\begin{equation}
 \modulusSymbol \mid ( 2^{2^n} - 1 ),
\end{equation}
i.e. $\modulusSymbol$ is a prime factor of $2^{2^n} - 1$, when the transform length $N$ is a power of two. The modulus is chosen to be prime so that one may divide the coefficients by any integer, allowing the construction of arbitrary exact filters. The value $2^{2^n} - 1$ must be the smallest multiple of $\modulusSymbol$ to the base 2 so that 2 is a $\lambda$-root of $\modulusSymbol$ as given by~\eqref{eqn::LambdaRoot} and \secTag~\ref{sec::Carmichael}.

The numbers of the form $2^{2^n} - 1$, which will be denoted as the Rader numbers, are a specific form of the Mersenne numbers
\begin{equation}
 M_n = 2^{n} - 1.
\end{equation}
Mersenne numbers can always be expressed as
\begin{equation}\label{eqn::Binomial}
2^{ab} - 1 = \left(2^{a} - 1\right)\left(1 + 2^a + 2^{2a} + \ldots + 2^{(b-1)a}\right),
\end{equation}
when $n$ is composite, since they are binomial numbers~\citep[pg. 42]{Schumer2004}. Applying this expansion to the Rader numbers 
\begin{align*}
&n = 1:  & 2^2-1 \\
&n = 2:  & 2^4-1 &= (2^2-1)\cdot (1 +2^2)\\
&n = 3:  & 2^8-1 &= (2^2-1)\cdot (1 +2^2 + 2^4 + 2^8)\\
&              &       &= (2^2-1)\cdot (1+2^2) \cdot (1+2^4)\\
&n = 4:  & 2^{16}-1 &= (2^2-1)\cdot (1 + 2^2 + 2^4 + 2^6 + \ldots + 2^{14})\\
&              &        &= (2^2-1)\cdot (1 +2^2 + 2^4 + 2^8) \cdot (1+2^8)\\
&              &        &= (2^2-1)\cdot (1+2^2) \cdot (1+2^4) \cdot (1+2^8),
\end{align*}
and so on until one arrives at an identity of the Fermat numbers~\citep[pg. 26]{Krizek2001}
\begin{equation}\label{eqn::Identity}
 2^{2^n} - 1 = F_0\cdot F_1 \cdot F_2 \cdot\ldots\cdot F_{n-1} = F_{n} - 2.
\end{equation}
For example, the first several factorisations of $2^{n} - 1$ are
\begin{align}
 2^{2} - 1 &= 3, & 2^{3} - 1 &= 7,\nonumber\\
 2^{4} - 1 &= 3\cdot 5, & 2^{5} - 1 &= 31,\nonumber\\
 2^{6} - 1 &= 3^2\cdot 7, & 2^{7} - 1 &= 127,\nonumber\\
 2^{8} - 1 &= 3\cdot 5\cdot 17, & 2^{9} - 1 &= 7\cdot 73,\nonumber\\
 2^{10} - 1 &= 3\cdot 11\cdot 31, & 2^{11} - 1 &= 23\cdot 89.\nonumber
\end{align}
\eqnTag~\eqref{eqn::Identity} suggests that $\modulusSymbol$ should be a Fermat number, noting that only the first five Fermat numbers are known to be prime.
\begin{theorem}[Fermat Number Moduli]\label{thm::FermatModuli}
The smallest power of two $\nu$ in \eqnTag~\eqref{eqn::LambdaRoot} for which
\begin{equation}\label{eqn::Rader}
 2^{2^n} \equiv 1 \imod \modulusSymbol,
\end{equation}
is true, is when the modulus $\modulusSymbol$ is the Fermat number $F_{n-1}$.
\end{theorem}
\begin{proof}
By the identity~\eqref{eqn::Identity}, higher order Fermat numbers can only become a factor of a given Rader number as $n$ increases. Multiples of these higher order Fermat numbers cannot exist as factors of Mersenne numbers $2^\ell-1$ less than the given Rader number since the only divisor of $2^n$ is two and powers of two by \eqnTag~\eqref{eqn::Binomial}, i.e. $\ell$ cannot be any number other than the divisors of $2^n$. Hence by \eqnTag~\eqref{eqn::Identity}, $2^{2^n}-1$ is the smallest multiple of $F_{n-1}$ to the base 2 so $N=2^n$.
\end{proof}
This results in a reformulation of the \ac{FNT}, which is only suitable for small transform lengths as the Fermat numbers grow large rapidly and are composite after $F_4$. Can one extend the above theorem to include the prime factors of large Fermat numbers? The answer is yes and it is the main theoretical result of this letter.
\begin{theorem}[Fermat Factor Moduli]\label{thm::FactorModuli}
The smallest power of two $\nu$ for which \eqnTag~\eqref{eqn::Rader} is true, is when the modulus $\modulusSymbol$ is a factor $r$ (prime or otherwise) of the Fermat number $F_{n-1}$.
\end{theorem}
\begin{proof}
Assume the contrary, that there exists a Mersenne number $2^\ell-1$ less than a given Rader number that is also a multiple of $r$. From \eqnTag~\eqref{eqn::Binomial}, $\ell \mid 2^n$, but the only divisors of $2^n$ is two or its powers. Now it is well known that the Fermat numbers do not share any common factor with each other, i.e. they are pair-wise coprime~\citep[pg. 63]{Schumer2004}. Hence, $\ell$ cannot be a power of two less than $2^n$ and so the first multiple of $2^\ell-1$ must be the Rader number. Thus, for $F_5 = 641\cdot 6700417$, $641$ or $6700417$ allows for $N=2^n$ when $n \leqslant 6$, since neither prime divides $2^\ell-1$ for all $\ell < 2^6$.
\end{proof}
We denote $r$ as a Rader prime when $r$ is prime. Some useful Rader primes are given in \tabTag~\ref{tab::Primes}, resulting in the transforms given in \eqnsTag~\eqref{eqn::NTT} and~\eqref{eqn::iNTT}. This is a far more useful result than \propTag~\ref{thm::FermatModuli} as now the moduli may be small or as large as desired, by simply selecting a Fermat number with a suitable small or large Rader prime. Although the factorisation of Fermat numbers is still an active area of research~\citep{Brent1996}, there are sufficient numbers of factors already known to accommodate any word size or for data hiding via large moduli.

Applying the Generalised Fermat and Mersenne Numbers to \propTag~\ref{thm::FactorModuli} should extend the \ac{NTT} of Dimitrov \emph{et al.}~\citep{Dimitrov1994}. Euler~\citep{Euler1732} showed that all prime factors of the Fermat numbers are of the form $k\cdot 2^{n+2}+1$, so \propTag~\ref{thm::FactorModuli} should also extend the work of Bhattacharya and Astola~\citep{Bhattacharya2000}. Finally, hardware implementations of this new transform may be similar to those constructed by McClellan~\citep{McClellan1976} and Leibowitz~\citep{Leibowitz1976} for the \ac{FNT}, since both require simple bit-shifting and prime factors of the Fermat numbers are of the form $k\cdot 2^{n+2}+1$~\citep{Euler1732}. See Agarwal and Burrus~\citep[\secTag VI.E]{Agarwal1974} for examples of how to compute the bit-shifting. The next section introduces an integer-only transform not requiring modulus operations.

\subsection{Modulus Free}
Carmichael~\citep{Carmichael1910} also proves the useful result
\begin{equation}\label{eqn::Dyadic}
a^{2^{\alpha-2}} \equiv 1 \imod {2^\alpha},
\end{equation}
where $a$ and 2 are coprime (see Appendix~\ref{sec::Totient}). This result can be used in constructing transforms that preserve the \ac{CCP}, while requiring no modulo operations in programming languages (such as $C$) or architectures that support ``wrap around'' upon overflow, i.e. the act of truncation is equivalent to modulo power of two (see \figTag~\ref{fig::Circles}(d)). This is advantageous because the integer division instruction, a critical part of the modulo operation, is generally a slow instruction. Note that an expression like \eqnTag~\eqref{eqn::Dyadic} is not possible using Euler's Theorem~\citep{Carmichael1910}.

Normalisation of the transform is also a concern, since the multiplicative inverse does not exist. This can be resolved by ensuring $2^\alpha$ is sufficiently large so that unnormalised values do not exceed $2^\alpha$. In other words, if $2^\beta$ is the bit depth of the data, then $\alpha \geqslant \beta + N$. 

The implementations of these transforms can be found in the NTTW C library~\citep{Chandra2009b}. Applications and performance comparisons of the various \acp{NTT}, as well as to the \ac{DFT}, will be part of a future publication.

\section*{Conclusion}\label{sec::Conclusion}
Transforms for fast digital convolutions were constructed that did not either require any multiplications or modulo operations (see \eqnsTag~\eqref{eqn::NTT},~\eqref{eqn::iNTT}, \tabTag~\ref{tab::Primes} and~\eqref{eqn::Dyadic}). The former utilises only bit-shifts, additions and modulo operations on prime factors of the Fermat numbers (denoted as Rader primes), while the latter only uses multiplications and additions. The result was made possible by using Carmichael's generalisation of Euler's Theorem, which also provides a more general theory of \aclp{NTT}.

\section*{Acknowledgements}
My thanks go to the Faculty of Science and School of Physics, Monash University and to Dr. I. Svalbe, Dr B. Parrein, Dr N. Normand and G. Ruben for their comments.

\appendices
\section{Congruences and Sino Notation}\label{sec::Sino}
Sino notation for computing the remainder or modulo operation is given as $\langle a \rangle_\modulusSymbol$, which denotes $a \!\imod \modulusSymbol$, i.e. $a = \langle a \rangle_\modulusSymbol + \modulusSymbol\cdot q$, with $a, \langle a \rangle_\modulusSymbol, q, \modulusSymbol \in \mathbb{Z}$ so that $0 \leqslant \langle a \rangle_\modulusSymbol < \modulusSymbol$. The modulo operation allows one to define a congruence, an example of which is given in \figTag~\ref{fig::Congruence}. Multiplicative inverses $b^{-1}$, i.e. the equivalent integers within congruences to do division by $b$ (thus turning division into a multiplication), are found using the efficient Extended Euclidean algorithm, provided they are coprime or their greatest common divisor is one, i.e. $\gcd(b,m)=1$.
\begin{figure}[htbp]
 \centering
 \includegraphics[width=0.25\textwidth]{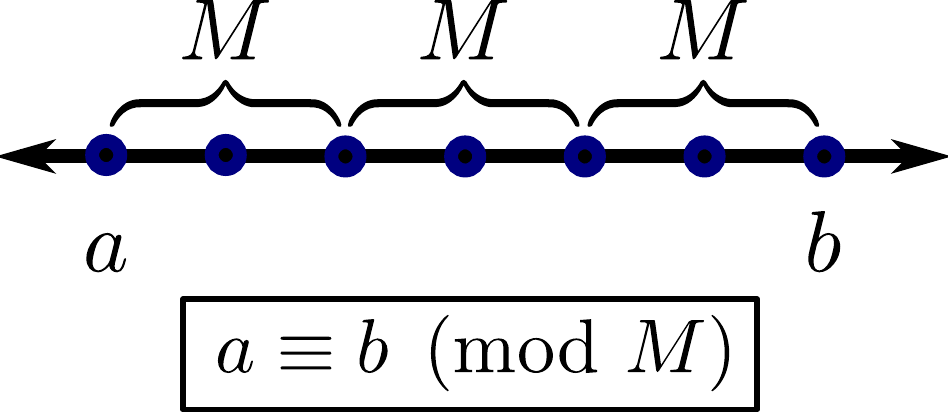} 
 \caption{Integer $a$ is congruent to integer $b$ when the distance between them is a multiple of $M$.}
 \label{fig::Congruence}
\end{figure}

\section{Totient \& Lambda Functions}\label{sec::Totient}
The Totient function $\phi(\modulusSymbol)$ is the number of integers less than $\modulusSymbol$ that do not have a common factor (or are coprime) with $\modulusSymbol$. For example, when $\modulusSymbol$ is prime, the function $\phi(\modulusSymbol) = \modulusSymbol-1$. The Lambda function $\lambda(\modulusSymbol)$ is defined in terms of the Totient function as follows
\begin{eqnarray}\label{eqn::LambdaDef}
\lambda(\modulusSymbol) =
 \begin{cases}
 \phi(\primeSymbol^{\alpha}), & \mbox{if } \modulusSymbol=\primeSymbol^{\alpha}\ \text{and}\ \primeSymbol\ \text{is an odd prime}\\ 
 \phi(2^{\alpha}), & \mbox{if } \modulusSymbol=2^{\alpha}\ \text{and}\ 0 \leqslant\alpha < 2\\
 \frac{1}{2}\phi(2^{\alpha}), & \mbox{if } \modulusSymbol=2^{\alpha}\ \text{and}\ \alpha > 2
 \end{cases}, \nonumber
\end{eqnarray}
so that $\lambda(2^{\alpha}\primeSymbol^{\alpha_1}_1\ldots \primeSymbol^{\alpha_j}_j)$ is the lowest common multiple of $\lambda(2^{\alpha}), \lambda(\primeSymbol^{\alpha_1}_1),\ldots, \lambda(\primeSymbol^{\alpha_j}_j)$ for $\modulusSymbol = 2^{\alpha}\primeSymbol^{\alpha_1}_1\ldots \primeSymbol^{\alpha_j}_j$. This allows one to define Carmichael's Theorem~\citep{Carmichael1910}
\begin{equation}\label{eqn::Carmichael}
 a^{\lambda(\modulusSymbol)} \equiv 1 \imod \modulusSymbol.
\end{equation}

\small
\bibliographystyle{IEEEtran}
\bibliography{RadonJabRef}

\acrodef{RT}{Rader Transform}
\acrodef{FT}{Fourier Transform}
\acrodef{FFT}{Fast Fourier Transform}
\acrodef{DFT}{Discrete Fourier Transform}
\acrodef{NTT}{Number Theoretic Transform}
\acrodef{CCP}{Circular Convolution Property}
\acrodef{CP}{Convolution Property}
\acrodef{FNT}{Fermat Number Transform}
\acrodef{MNT}{Mersenne Number Transform}
\acrodef{FTL}{Finite Transform Library}
\acrodef{SSE}{Streaming SIMD Extensions}
\acrodef{CPU}{Central Processing Unit}

\end{document}